\documentclass{amsart}
\usepackage{amssymb,delimset,siunitx,booktabs}

\usepackage[normalem]{ulem}

\usepackage[mathscr]{euscript}

\usepackage{enumitem}
\setlist{itemsep=4pt, topsep=0pt, leftmargin=17pt}

\usepackage{mathtools}

\DeclarePairedDelimiter\floor{\lfloor}{\rfloor}

\usepackage{xurl}

\usepackage{xcolor}

\definecolor{PKU}{cmyk}{0, 1, 1, .45}
\definecolor{BIT}{cmyk}{1, 0, 1, 0}
\usepackage[colorlinks, citecolor=BIT, linkcolor=PKU, pagebackref]{hyperref}
\newcolumntype{L}{>{$}l<{$}} 
\newcolumntype{R}{>{$}r<{$}} 
\newcolumntype{C}{>{$}c<{$}} 


\usepackage[capitalize]{cleveref}
\crefname{itm}{}{}
\creflabelformat{itm}{~\upshape(#2#1#3)}
\crefname{def}{Def.}{Defs.}
\Crefname{def}{Definition}{Definitions}
\creflabelformat{def}{~\upshape(#2#1#3)}
\crefname{ineq}{Ineq.}{Ineqs.}
\Crefname{ineq}{Inequality}{Inequalities}
\creflabelformat{ineq}{~\upshape(#2#1#3)}
\crefname{step}{Step}{Step}
\creflabelformat{step}{~\upshape(#2#1#3)}

\AddToHook{env/conjecture/begin}{\crefalias{theorem}{conjecture}}
\AddToHook{env/lemma/begin}{\crefalias{theorem}{lemma}}
\AddToHook{env/proposition/begin}{\crefalias{theorem}{proposition}}

\makeatletter
\newcommand\creflabel[2][\@currentcounter]{%
 \crefalias{\@currentcounter}{#1}\label{#2}}
\makeatother

\newtheorem{theorem}{Theorem}[section]

\newtheorem{lemma}[theorem]{Lemma}

\newtheorem{proposition}[theorem]{Proposition}
\numberwithin{equation}{section}

\allowbreak
\allowdisplaybreaks

\usepackage[numbers, sort&compress, nonamebreak, merge, elide, longnamesfirst]{natbib}

\usepackage{tikz}
\usetikzlibrary{decorations.pathreplacing, calc, positioning, arrows.meta}
\tikzset{
vertex/.style={shape=circle, minimum size=0.9mm, ball color=black, inner sep=0.5},
edge/.style={black, very thick},
ball/.style={shape=circle, ball color=black, minimum size=1mm, inner sep=0.5},
ellipsis/.style={shape=circle, fill, inner sep=.5},
range.lr/.style={|{Stealth}-{Stealth}|, thin},
range.l/.style={-{Stealth}|, thin},
range.r/.style={|{Stealth}-, thin},
EEdge/.style={very thick, color=black},
apball/.style={ball color=black}
}

\author[S.Y.M. Gong]{Simon Y.M. Gong}
\address[Simon Y.M. Gong]{School of Mathematics and Statistics, Beijing Institute of Technology, Beijing 102400, P. R. China.}
\email{simon@bit.edu.cn}

\author[D.G.L. Wang]{David G.L. Wang}
\address[David G.L. Wang]{School of Mathematics and Statistics \& MIIT Key Laboratory of Mathematical Theory and Computation in Information Security, Beijing Institute of Technology, Beijing 102400, P. R. China.}
\email{glw@bit.edu.cn}

\author[K. Zhang]{K. Zhang}
\address[K. Zhang]{School of Mathematics and Statistics, Beijing Institute of Technology, Beijing 102400, P.\ R.\ China.}
\email{kai@bit.edu.cn}

\thanks{David Wang is supported by the National Natural Science Foundation of China (Grant No.~12171034).}


\title{The trinacria graphs $T_{(b+2)b2}$ are $e$-positive}

\begin{document}

\begin{abstract}
In this paper, we identify a new family of $e$-positive graphs, called the trinacria graphs $T_{(b+2)b2}$, thereby providing a partial answer to Stanley's question on which graphs are $e$-positive. The trinacria graph $T_{abc}$ is the graph on $a+b+c+3$ vertices obtained by attaching paths $P_a$, $P_b$ and~$P_c$ to the vertices of a triangle, respectively. Our proof relies on several ad hoc combinatorial ideas, and employs divide-and-conquer techniques, charging arguments, and progressive repair methods.
\end{abstract}
\maketitle
\tableofcontents

\section{Introduction}

In 1995,  \citet{Sta95} introduced the concept of the \emph{chromatic symmetric function} $X_G$ of a graph~$G$, which encodes various types of enumerative information including the number of vertices, edges, and triangles, as well as the girth, the matching polynomial, and more. The positivity problem for symmetric functions has received much attention because of its connections with mathematical physics and representation theory. In this context, \citet{Sta95} reformulated a conjecture of \citet{SS93} as: \emph{The chromatic symmetric function of the incomparability graph of any $(3+1)$-free poset is $e$-positive.} This conjecture has been called \emph{Stanley--Stembridge's conjecture}, and has been one of the most influential conjectures in algebraic combinatorics over the past three decades. It was recently confirmed by \citet{Hik25}. \citet[Section 5]{Sta95} proved that paths and cycles are $e$-positive and posed the following question:
\begin{center}
``\emph{Which $X_G$ is $e$-positive?}''
\end{center}
Our work continues in this direction.

A typical graph family arising in Stanley--Stembridge's conjecture is the \emph{unit interval graphs}, which includes \emph{complete graphs}, \emph{paths}, \emph{lariats}, \emph{lollipops} and \emph{melting lollipops}, \emph{dumbbells}, \emph{generalized bulls}, \emph{$K$-chains}, \emph{triangular ladders}, among others; see \citet{HNY20,GS01,SW16,Tom25-CT,Dv18,Ale21}. \citet{Tom24} obtained a signed $e$-expansion of the chromatic symmetric function of every unit interval graph, thereby implying the $e$-positivity of \emph{melting $K$-chains}.

There is also a substantial collection of $e$-positive graphs that are not unit interval graphs. A large family of such graphs was identified by \citet{TV25X}, who proved that gluing a sequence of unit interval graphs and cycles results in an $e$-positive graph, and that any graph obtained by identifying the first and last vertices of such a sequence is also $e$-positive. Such graphs include \emph{cycles}, \emph{tadpoles}, \emph{hats}, \emph{adjacent cycle chains} (or \emph{hat-chains}), \emph{kayak paddles}, \emph{clique-path-cycle graphs}, \emph{cycle-chords}, and \emph{noncrossing cycle-chords}; see \citet{TV25,QTW25,WZ25,Wang25,TW24X}. Further examples of $e$-positive graphs outside these classes include the \emph{line graphs of tadpoles}, \emph{twinned paths}, \emph{twinned lollipops}, \emph{twinned cycles}, \emph{spiders} of the form $S(b+1,b,1)$ or $S(4m+2,\,2m,\,1)$, and \emph{clocks}; see \citet{BCCCGKKLLS25,TW24X,WW23-JAC,DFv20,TWW24X,CHW26}. The second author~\cite{Wang25} conjectures that all \emph{theta graphs} are $e$-positive.

Some graphs have been shown to be not $e$-positive, or even not Schur positive. A leading conjecture in this direction is due to \citet{DSv20}: \emph{no tree with maximum degree $\Delta\ge4$ is $e$-positive}. It was proved by \citet{Zhe22} for trees with $\Delta\ge 6$, and then by \citet{Tom25-AAM} for trees with $\Delta\ge5$ and spiders with $\Delta=4$. Other graph families that have been shown to be not $e$-positive include \emph{generalized nets}, \emph{saltires} $\mathrm{SA}_{nn}$, \emph{augmented saltires} $\mathrm{AS}_{nn}$ and $\mathrm{AS}_{n(n+1)}$, \emph{triangular towers}~$\mathrm{TT}_{nnn}$, \emph{brooms} except $\mathrm{br}(2,2)$, and \emph{double brooms}; see \citet{DFv20,DSv20,WW23-DAM}. Since the $e$-positivity implies Schur positivity, attention has also been paid to graphs that are not Schur positive; see \citet{Sta98,WW20,LLYZ25,WZ25X}.

In this paper, we study \emph{trinacria graphs} $T_{abc}$, defined as the graph on $a+b+c+3$ vertices obtained by attaching the paths $P_a$, $P_b$, and $P_c$ to the vertices of a triangle, see \cref{fig:trinacria1}.
\begin{figure}[htbp]
\centering
\begin{tikzpicture}[scale=1.5]
\coordinate (a1) at (xyz polar cs:angle=0, radius=.5);
\coordinate (a2) at (xyz polar cs:angle=90, radius=.866);
\coordinate (a3) at (xyz polar cs:angle=180, radius=.5);
\coordinate (b1) at (xyz polar cs:angle=0, radius=1.1);
\coordinate (b2) at (xyz polar cs:angle=0, radius=1.25);
\coordinate (b3) at (xyz polar cs:angle=0, radius=1.4);
\coordinate (c) at (xyz polar cs:angle=0, radius=2);
\coordinate (d) at (0+1.5, .866);
\coordinate (f) at (-2, 0);
\coordinate (e1) at (0.6, .866);
\coordinate (e2) at (0.75, .866);
\coordinate (e3) at (0.9, .866);
\coordinate (g1) at (-1.4, 0);
\coordinate (g2) at (-1.25, 0);
\coordinate (g3) at (-1.1, 0);
\draw[edge] (a1) -- (a2) -- (a3) -- (a1);
\draw[edge] (a1) -- (0.9, 0);
\draw[edge] (a2) -- (0.4, .866);
\draw[edge] (a3) -- (-0.9, 0);
\draw[edge] (d) -- (1.1, .866);
\draw[edge] (c) -- (1.6, 0);
\draw[edge] (f) -- (-1.6, 0);
\foreach \e in {a1, a2, a3, c, d, f}
  \shade[ball](\e)  circle (.035cm);
\foreach \f in {b1, b2, b3, e1, e2, e3, g1, g2, g3}
  \filldraw[ellipsis] (\f) circle (.015cm);
\draw (-1.25, -.3) node {$P_a$};
\draw (1.3, -.3) node {$P_b$};
\draw (0.9, -.4+1) node {$P_c$};
\end{tikzpicture}
\caption{The trinacria graph $T_{abc}$.}
\label{fig:trinacria1}
\end{figure}
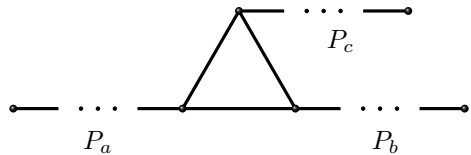
In particular, the graphs~$T_{ab0}$ is the hat graph $H_{a3b}$, and the graph ~$T_{a00}$ is the lariat graph $L_{a+3}$; see \cite{WZ25,Dv18}. The trinacria $T_{111}$, known as the \emph{net}, is the smallest claw-free graph that is not $e$-positive, and hence cannot be the incomparability graph of any $(3+1)$-free poset; see \cite{Sta95}. Consequently, no trinacria $T_{abc}$ is the incomparability graph of a $(3+1)$-free poset, since each contains the net as an induced subgraph. On the other hand, trinacria graphs are special cases of \emph{suns} and \emph{generalized spiders}, which have been investigated by \citet{sV25,FKKMMT20}. \Citet[Theorem~3.14]{MS25} established that $T_{abc}$ is not $e$-positive when the lengths $a$, $b$, and~$c$ satisfy the triangle inequality. In contrast, we prove the following $e$-positivity result:
\begin{theorem}\label{thm:e+.trinacria1}
Every trinacria graph of the form $T_{(b+2)b2}$ is $e$-positive.
\end{theorem}
As an immediate corollary, this shows that \citeauthor{MS25}'s result is sharp.

Several methods for proving the $e$-positivity of a graph have been developed. Notable examples include the generating function method of \citet{Sta95}, the set partition method of \citet{GS01} for $(e)$-positivity, the basis transformation method of \citet{LY21}, the composition method of \citet{WZ25}, the forest triple method of \citet{Tom24}, and the Young tableau method of \citet{Hik25}. In this paper, we prove \cref{thm:e+.trinacria1} using the composition method. A key advantage of the composition method lies in establishing the $e$-positivity of a symmetric function: for any $a>b>0$ and compositions $K$ and $H$ with the same underlying partition, we have $ae_K-be_H=(a-b)e_K$. 
This combination facilitates the construction of a positive $e_I$-expansion of a symmetric function, thereby proving its $e$-positivity. On the other hand, an application of the composition method often requires additional care for combining terms in an $e_I$-expansion that share the same underlying partition. Our proof relies on several ad hoc combinatorial ideas, and employs \emph{divide-and-conquer} techniques, \emph{charging arguments}, and \emph{progressive repair methods}. Details are provided in \cref{sec:pf}.

This paper is organized as follows. In \cref{sec:pre}, we recall necessary concepts and known results. In \cref{sec:pf}, we provide a proof of \cref{thm:e+.trinacria1}. We write $X_{T_{(b+2)b2}}=Y_2e_1^2+Y_1e_1+Y_0$, where each symmetric function $Y_i$ is free of a factor $e_1$, and prove that each symmetric function $Y_i$ is $e$-positive. For $Y_2$ and $Y_1$, we apply charging arguments; for $Y_0$, we employ a progressive repair method, in which each round addresses the remaining negative contributions.

\section{Preliminaries}\label{sec:pre}
This section provides necessary background on chromatic symmetric functions. We adopt the terminology system from \citet{GKLLRT95} and \citet{Sta11B}. Let $n$ be a positive integer. Denote $[n]=\{1,2,\dots,n\}$. A \emph{composition} of $n$ is a sequence of positive integers whose sum is~$n$, denoted by $I=i_1 \dotsm i_z\vDash n$, with \emph{size} $\abs{I}=n$, \emph{length} $\ell(I)=z$, and \emph{parts} $i_1,\dots,i_z$. When all parts $i_k$ are equal to a common value $i$, we write $I=i^z$. When a capital letter like $I$ or~$J$ stands for a composition, its small letter counterpart with integer subscripts refers to the parts. A \emph{partition}\index{partition} of~$n$ is a multiset of positive integers $\lambda_i$ whose sum is~$n$, denoted $\lambda=\lambda_1\lambda_2\dotsm\vdash n$, where $\lambda_1\ge \lambda_2\ge\dotsm\ge 1$.

A \emph{symmetric function} of homogeneous degree $n$ over the field $\mathbb{Q}$ of rational numbers is a formal power series
\[
f(x_1, x_2, \dots)
=\sum_{\lambda=\lambda_1 \lambda_2 \dotsm \vdash n}
c_\lambda \cdotp x_1^{\lambda_1} x_2^{\lambda_2} \dotsm
\]
such that $f(x_1, x_2, \dots)=f(x_{\pi(1)}, x_{\pi(2)}, \dots)$ for any permutation $\pi$. Let $\operatorname{Sym}^0=\mathbb{Q}$, and let $\operatorname{Sym}^n$ be the vector space of homogeneous symmetric functions of degree $n$ over~$\mathbb{Q}$. One basis of $\operatorname{Sym}^n$ consists of the elementary symmetric functions $e_\lambda$ for all partitions $\lambda\vdash n$, where
\[
e_\lambda=e_{\lambda_1}e_{\lambda_2}\dotsm
\quad\text{and}\quad
e_k=\sum_{1\le i_1<\dots<i_k} x_{i_1} \dotsm x_{i_k}.
\]
A symmetric function $f\in\mathrm{Sym}$ is said to be \emph{$e$-positive} if every $e_\lambda$-coefficient of $f$ is nonnegative. \Citet{Sta95} introduced the \emph{chromatic symmetric function} for a graph $G=(V,E)$ as
\[
X_G
=\sum_{\kappa\colon V\to\{1,2,\dots\}}
\prod_{v\in V} x_{\kappa(v)},
\]
where $\kappa$ runs over proper colorings of~$G$. 

\citet[Theorem 3.1, Corollaries 3.2 and 3.3]{OS14} established the \emph{triple-deletion} property for chromatic symmetric functions, which is analogous to the \emph{deletion-contraction} recursion for Birkhoff's \emph{chromatic polynomials}; see \citet{DKT05B}. It is widely used in reducing the computation of chromatic symmetric functions to that of smaller graphs; see \cref{prop:X.spider.abc} for example.

\begin{proposition}[\citeauthor{OS14}]\label{prop:3del}
Let $G$ be a graph with a stable set $T$ of order $3$. Denote by $e^1$, $e^2$ and $e^3$ the edges linking the vertices in $T$. For any set $S\subseteq \{1,2,3\}$, denote by $G_S$ the graph with vertex set~$V(G)$ and edge set $E(G)\cup\{e^j\colon j\in S\}$. Then 
\[
X_{G_{12}}=X_{G_1}+X_{G_{23}}-X_{G_3}
\quad\text{and}\quad
X_{G_{123}}=X_{G_{13}}+X_{G_{23}}-X_{G_3}.
\]
\end{proposition}

For any composition $I$, we define $e_I=e_{\rho(I)}$, where $\rho(I)$ is the underlying partition of $I$. An \emph{$e_I$-expansion} of a symmetric function $f\in\mathrm{Sym}^n$ is an expression
\[
f=\sum_{I\vDash n} c_I e_I.
\]
We call it a \emph{positive $e_I$-expansion} of $f$ if $c_I\ge 0$ for all $I$. \citet[Table~1]{SW16} discovered a positive $e_I$-expansion for paths $P_n$.

\begin{proposition}[\citeauthor{SW16}]\label{prop:X.path}
We have $X_{P_n}=\sum_{I\vDash n}w_Ie_I$, where 
\[
w_I=i_1(i_2-1)(i_3-1)\dotsm(i_z-1)
\quad\text{for $I=i_1i_2\dotsm i_z$}.
\]
\end{proposition}

A \emph{spider} is a tree formed by several paths that share a common endpoint. More precisely, for any partition $\lambda=\lambda_1\dotsm\lambda_d\vdash n-1$, the spider $S(\lambda)$ is the $n$-vertex tree consisting of the paths $P_{1+\lambda_1}$, $\dots$, $P_{1+\lambda_d}$ which all meet at a single vertex of degree $d$. \Citet[Lemma 4.4]{Zhe22} expressed the chromatic symmetric function of $S(abc)$ in terms of partial convolutions of the chromatic symmetric functions of paths, using \cref{prop:3del}.

\begin{proposition}[\citeauthor{Zhe22}]\label{prop:X.spider.abc}
If $abc\vdash n-1$, then 
\[
X_{S(abc)}
=X_{P_n}+\sum_{i=1}^c \brk1{X_{P_i}X_{P_{n-i}}-X_{P_{b+i}} X_{P_{n-b-i}}}.
\]
\end{proposition}

We consider the free associative algebra $\mathcal F=\bigoplus_{n\ge 0}\mathcal V_n$, where $\mathcal V_n=\mathbb Q\brk[a]{e_I\colon I\vDash n}$ and $e_Ie_J=e_{I\!J}$. Then the vector space $\mathrm{Sym}^n$ is embeddable into $\mathcal V_n$ naturally since any partition can be considered as a composition with parts  from large to small. For any element $f\in\mathcal F$ and any composition~$K$, we use the notation $f\bigl|_K$ to denote the coefficient of~$K$ in~$F$. For any set $\mathcal C$ of compositions, denote 
\[
f\big|_{\mathcal C}=\sum_{K\in\mathcal C}f\big|_K.
\]
We often consider $f\in\mathcal F$ as a symmetric function by treating $e_I$ as $e_{\rho(I)}$. Denote by $[e_\lambda]f$ the $e_\lambda$-coefficient of $f$, for any partition $\lambda$. For instance, for $f=ae_{211}+be_{121}+ce_{13}\in\mathcal V_4$, 
\[
f\bigl|_{211}=a
\quad\text{and}\quad
[e_{211}]f=f\big|_{\{211, \,121,\,112\}}=a+b.
\]

\section{Proof of the $e$-positivity of $T_{(b+2)b2}$}\label{sec:pf}

This section consists of a proof of \cref{thm:e+.trinacria1}. First of all, we express $X_{T_{abc}}$ in terms of partial convolutions of chromatic symmetric functions of paths.

\begin{proposition}\label{prop:rec:3sun}
For any $a\ge b\ge c\ge 1$, 
\begin{equation}\label{Tabc:path}
X_{T_{abc}}
=2\mathscr P_0
+\sum_{i=1}^{c+1}(\mathscr P_i-\mathscr P_{b+i})
+\sum_{i=1}^c(\mathscr P_i-\mathscr P_{b+i+1})
-\mathscr P_{a+1},
\end{equation}
where $\mathscr P_i=X_{P_i}X_{P_{n-i}}$.
\end{proposition}
\begin{proof}
Let $G=T_{abc}$ and $n=a+b+c+3$. Applying \cref{prop:3del}, we obtain
\[
X_G=X_{S(a+1,\,b,\,c+1)}+X_{S(a+1,\,b+1,\,c)}-\mathscr P_{a+1}.
\]
Using \cref{prop:X.spider.abc}, it is routine to derive the desired formula.
\end{proof}

Second, we produce an $e_I$-expansion for the chromatic symmetric function of the trinacria $T_{(b+2)b2}$. Denote by $\mathcal W_n$ the set of compositions without parts $1$, namely,
\[
\mathcal W_n=\{I\vDash n\colon i_1,i_2,\dots\ge 2\}.
\]
Let $\mathcal W_n(p)$ be the set of compositions in $\mathcal W_n$ with a prefix of size~$p$. The \emph{indicator function} $\chi(\cdot)$ is defined for propositions $P$ by 
\[
\chi(P)=\begin{dcases*}
1,& if $P$ is true,\\
0,& otherwise.
\end{dcases*}
\]

\begin{proposition}\label{prop:X.3sun:b+2.b.2}
We have $X_{T_{(b+2)b2}}=Y_2e_1^2+Y_1e_1+Y_0$, where
\begin{align}
\creflabel[def]{def:Y2}
Y_2&=\sum_{K\in\mathcal W_{n-2}}\brk1{2+\chi(k_1=2)} w_{1K} e_K
-\sum_{l=1}^3\sum_{K\in\mathcal W_{2b+5}(b+6-l)} l\cdot w_{1K} e_K,\\
\creflabel[def]{def:Y1}
Y_1&=Y_{11}+Y_{12}+Y_{13}
-\sum_{l=1}^6\sum_{K\in\mathcal W_{2b+6}(b+l)}\min(l,\,7-l) w_K e_K, 
\quad\text{and}\\
\creflabel[def]{def:Y0}
Y_0&=\sum_{K\in\mathcal W_{2b+7}} f(K) e_K\prod_{i\ge 1}(k_i-1).
\end{align}
Here
\begin{align*}
Y_{11}&=\sum_{K\in\mathcal W_{2b+6}}(2w_K+w_{2K})e_K, 
\quad
Y_{12}=\sum_{I\in\mathcal W_{2b+4}}(w_{4I}+w_I)e_{2I}, 
\quad
Y_{13}=\sum_{I\in\mathcal W_{2b+3}} w_{3I} e_{3I},
\end{align*}
and for $r_i=i/(i-1)$,
\begin{equation}\creflabel[def]{def:f}
f(K)=2r_{k_1}+4r_{k_2}\chi(k_1=2)+\frac{3}{2}r_{k_2}\chi(k_1=3)
-r_{k_1}\sum_{l=1}^3 l\cdotp r_{j_1}\cdotp \chi(\text{$K=I\!J$ with $J\vDash b+l$}).
\end{equation}
\end{proposition}
\begin{proof}
Let $G=T_{(b+2)b2}$ and $n=2b+7$. By \cref{prop:rec:3sun},
\[
X_G=2X_{P_n}+2\mathscr P_1+2\mathscr P_2+\mathscr P_3
-\mathscr P_{b+1}-2\mathscr P_{b+2}-3\mathscr P_{b+3}.
\]
Since $X_{P_1}=e_1$, $X_{P_2}=2e_2$, $X_{P_3}=3e_3+e_{21}$ and
\[
\mathscr P_i=X_{P_i}X_{P_{n-i}}
=\sum_{I\vDash i,\ J\vDash n-i} w_I w_J e_{I\!J},
\]
we find
\begin{align*}
X_G&=\sum_{K\vDash 2b+7} 2w_K e_K
+\sum_{I\vDash 2b+6} 2w_I e_{1I}
+\sum_{I\vDash 2b+5} 4w_I e_{I2}
+\sum_{I\vDash 2b+4}(3w_I e_{I3}+w_I e_{1I2})\\
&\qquad
-\sum_{l=1}^3\sum_{I\vDash b+7-l, \, J\vDash b+l} 
l\cdotp w_I w_J e_{I\!J}.
\end{align*}
Sorting it according to the power of $e_1$, we obtain the desired formula for $Y_2$ and $Y_1$, and 
\[
Y_0=\sum_{K\in\mathcal W_{2b+7}} 2w_K e_K
+\sum_{I\in\mathcal W_{2b+5}} 4w_I e_{2I}
+\sum_{I\in\mathcal W_{2b+4}} 3w_I e_{3I}
-\sum_{l=1}^3 \sum_{(I,J)\in\mathcal W_{b+7-l}\times\mathcal W_{b+l}}
l\cdotp w_I w_J e_{I\!J}.
\]
Expressing $Y_0$ in terms of $r_i$, we obtain \cref{def:Y0}.
\end{proof}

By \cref{prop:X.3sun:b+2.b.2}, we divide the $e$-positivity problem into those for $Y_2$, $Y_1$ and $Y_0$. We shall conquer them in \cref{sec:Y2,sec:Y1,sec:Y0}, respectively.

\subsection{$Y_2$ is $e$-positive}\label{sec:Y2}
We come up with a positive $e_I$-expansion of $Y_2$.

\begin{lemma}\label{lem:e+:trinacria1.Y2}
The symmetric function $Y_2$ defined by \cref{def:Y2} has the positive $e_I$-expansion
\[
Y_2=\sum_{K\in\mathcal A\cup \mathcal B} w_{1K}e_K
+\sum_{K\in\mathcal C}\brk1{1+\chi(k_1=2)} w_{1K}e_K
+\sum_{K\in\mathcal D}\brk1{2+\chi(k_1=2)} w_{1K}e_K, 
\]
where 
\begin{align*}
\mathcal A&=
\{I\!J\in\mathcal W_{2b+5}\colon\abs{I}=b+3,\ i_1\ge 4\},\\
\mathcal B&=
\{I\!J\in\mathcal W_{2b+5}\colon\abs{I}=b+4,\ i_1=2,\ i_{-1}\ge 3\},\\
\mathcal C&=
\{I\!J\in\mathcal W_{2b+5}\colon\abs{I}=b+5,\ i_{-1}\ge 4\}, 
\end{align*}
and $\mathcal D$ is the set of compositions in $\mathcal W_{2b+5}$ with no prefix of size $b+\delta$ for any $\delta\in\{2, 3, 4, 5\}$.
\end{lemma}
\begin{proof}
Let $\mathcal W=\mathcal W_{2b+5}$.
For $K\in\mathcal W$, let $L_K$ be the set of numbers $l\in[3]$ such that $K$ has a prefix of size $b+6-l$. By \cref{def:Y2}, 
\begin{equation}\label{pf:Y2}
Y_2=\sum_{K\in\mathcal W}c_K w_{1K}e_K,
\end{equation}
where
\begin{equation}\label{pf:f.Y2}
c_K=2+\chi(k_1=2)-\sum_{l\in L_K}l.
\end{equation}
We can decompose~$\mathcal W$ into $11$ pairwise disjoint sets:
\[
\mathcal W
=\mathcal A_2
\cup
(\mathcal A_{30}\cup\mathcal A_{31}\cup\mathcal A_{32}\cup\mathcal A)
\cup
(\mathcal A_{40}\cup \mathcal A_{41}\cup\mathcal B)
\cup
(\mathcal A_5\cup\mathcal C)
\cup
\mathcal D,
\]
where
\begin{align*}
\mathcal A_2
&=\{I\!J\in\mathcal W\colon\abs{I}=b+2,\ j_1\ge 4\},\\
\mathcal A_{30}
&=\{I\!J\in\mathcal W\colon\abs{I}=b+3,\ i_1=2,\ j_1\ge 3\},\\
\mathcal A_{31}
&=\{I\!J\in\mathcal W\colon\abs{I}=b+3,\ i_1=2,\ j_1=2\},\\
\mathcal A_{32}
&=\{I\!J\in\mathcal W\colon\abs{I}=b+3,\  i_1=3\},\\
\mathcal A_{40}
&=\{I\!J\in\mathcal W\colon\abs{I}=b+4,\ i_1\ge 3\},\\
\mathcal A_{41}
&=\{I\!J\in\mathcal W\colon\abs{I}=b+4,\ i_1=2,\ i_{-1}=2\},\quad\text{and}\\
\mathcal A_5
&=\{I\!J\in\mathcal W\colon\abs{I}=b+5,\ i_{-1}=3\}.
\end{align*}
Consider the map $\varphi\colon\mathcal W(b+3)\to\mathcal W(b+2)$
defined by $\varphi(I\!J)=J\!I$, where $\abs{I}=b+3$. It is clear that~$\varphi$ is a bijection. Moreover, 
\[
\varphi(\mathcal A_{31})=\mathcal A_{41},
\quad
\varphi(\mathcal A_{32})=\mathcal A_5,
\quad\text{and}\quad 
\varphi(\mathcal A)=\mathcal A_2.
\]
Let $H=\varphi(K)$. Then $w_{1H}e_H=w_{1K}e_K$ as symmetric functions. By \cref{pf:f.Y2}, 
\[
c_K+c_H=\begin{dcases*}
0,& if $K\in\mathcal A_{31}\cup \mathcal A_{32}$,\\
1,& if $K\in\mathcal A$,
\end{dcases*}
\quad\text{and}\quad
c_K=\begin{dcases*}
0,& if $K\in\mathcal A_{30}\cup \mathcal A_{40}$,\\
1,& if $K\in\mathcal B$,\\
1+\chi(k_1=2),& if $K\in\mathcal C$,\\
2+\chi(k_1=2),& if $K\in\mathcal D$.
\end{dcases*}
\]
Substituting them into \cref{pf:Y2}, we obtain the desired formula.
\end{proof}

\subsection{$Y_1$ is $e$-positive}\label{sec:Y1}
In this section, we let $L_K$ be the set of numbers $l\in[6]$ such that $K$ has a prefix of size $b+l$. Let 
\[
l_K=\sum_{l\in L_K}\min(l,\,7-l)
\quad\text{and}\quad
N=\sum_{K\in\mathcal W_{2b+6}} l_K w_K e_K.
\]
Then $Y_1=Y_{11}+Y_{12}+Y_{13}-N$.
For any set $S\in 2^{[6]}$, let $\mathcal E_S$ be the set of compositions $K\in\mathcal W_{2b+6}$ such that $L_K=S$. Then $\mathcal E_S=\emptyset$ if $\{i,\, i+1\}\subseteq S$ for any $i\in[5]$. We write $\mathcal E_{\{s_1,\,\dots,\,s_r\}}=\mathcal E_{s_1\dotsm s_r}$ for compactness. Then
\[
\bigcup_{i=1}^6\mathcal W_{2b+6}(b+i)
=\brk4{\bigsqcup_{i=1}^6 \mathcal E_i}
\bigsqcup
\brk4{\bigsqcup_{\substack{1\le i<j\le 6\\
j\ge i+2}}\mathcal E_{ij}}
\bigsqcup
\brk1{\mathcal E_{135}\sqcup\mathcal E_{136}
\sqcup\mathcal E_{146}\sqcup\mathcal E_{246}}.
\]
The $20$ pairwise disjoint subsets on the right side constitute the set of compositions that appear in $N$.

Suppose that $\abs{S}\le 1$ or $S\in\{\{1, 5\}, \,\{1, 6\}, \,\{2, 6\}\}$. For any composition $K\in\mathcal E_S$,
\[
(Y_{11}-N)\big|_K=(2-l_K)w_K+w_{2K}\ge -w_K+w_{2K}\ge 0.
\]
We decompose the union of the remaining $11$ sets into the following $4$ families:
\[
\mathcal F_3=\mathcal E_{13},\quad
\mathcal F_4=\mathcal E_{14}\cup\mathcal E_{24},\quad
\mathcal F_5=\mathcal E_{25}\cup\mathcal E_{35}\cup\mathcal E_{135},
\quad\text{and}\quad
\mathcal F_6=\mathcal E_{36}\cup\mathcal E_{46}
\cup\mathcal E_{136}\cup\mathcal E_{146}\cup\mathcal E_{246},
\]
according to the maximum number in the subscripts. Recall that $\rho(K)$ is the underlying partition of~$K$. Our strategy for showing the $e$-positivity of $Y_1$ is the following \emph{charging argument}:

\begin{description}
\item[Step 1. Grouping]
Decompose each family $\mathcal F_k$ into pairwise disjoint subfamilies
\[
\mathcal A(I,J;\theta)\subseteq\{\alpha I\beta J\in\mathcal F_k\colon \rho(\alpha\beta)=\theta\},
\]
where $I$ and $J$ are compositions such that $\abs{J}=b+6-k$, $\theta$ is a partition such that $\rho(\alpha\beta)=\theta$, and each of the factors~$\alpha$ and $\beta$ is allowed to be empty. Namely,
\begin{equation}\label{decomp:Y1.Fk}
\mathcal F_k=\bigsqcup\mathcal A(I,J;\theta),
\end{equation}
where the number of triples $(I,J;\theta)$ in the decomposition above is finite.
\item[Step 2. Matching]
For any set $\mathcal A=\mathcal A(I,J;\theta)$,
let~$\mathcal A(i)$ be the set of compositions that are obtained from a composition $\alpha I\beta J\in\mathcal A$ by removing a part~$i$ from $\alpha$ or $\beta$. Find two sets $\mathcal B\subseteq \mathcal A(2)$ and $\mathcal C\subseteq\mathcal A(3)$ such that $s(\mathcal A,\mathcal B,\mathcal C)\ge 0$, where
\[
s(\mathcal A,\mathcal B,\mathcal C)
=(Y_{11}-N)\big|_{\mathcal A}
+Y_{12}\big|_{\{2H\colon H\in\mathcal B\}}
+Y_{13}\big|_{\{3H\colon H\in\mathcal C\}}.
\]
\item[Step 3. Distinctness verification]
Show that the sets $\mathcal B$ for distinct sets $\mathcal A$ are pairwise disjoint, and that the sets $\mathcal C$ for distinct sets $\mathcal A$ are pairwise disjoint.
\end{description}

First of all, we produce a lemma for fast computing the sum $s(\mathcal A,\mathcal B,\mathcal C)$.

\begin{lemma}\label{lem:computation.Y1}
Fix a partition $\theta$ and two compositions $I$ and $J$.
Let $\mathcal A\subseteq\{\alpha I\beta J\in\mathcal W_{2b+6}\colon \rho(\alpha\beta)=\theta\}$. For any composition $K=\alpha I\beta J\in\mathcal A\cup\mathcal A(2)$, let $K(i_1)$ be the composition obtained by replacing the factor $I$ with the part $i_1$ and by removing the suffix $J$. Then for any sets $\mathcal B\subseteq \mathcal A(2)$ and $\mathcal C\subseteq \mathcal A(3)$, we have $s(\mathcal A,\mathcal B,\mathcal C)=c(\mathcal A, \mathcal B, \mathcal C)w_{I\!J}'$, where
\[
c(\mathcal A, \mathcal B, \mathcal C)=
\sum_{K\in\mathcal A}(2-l_K)w_{K(i_1)}
+\sum_{H\in\mathcal B}w_{H(i_1)}
+\brk3{2\abs{\mathcal A}+4\abs{\mathcal B}+\frac{3}{2}\abs{\mathcal C}}
(i_1-1)w_{1\theta}
\]
and $w_{k_1\dotsm k_z}'=(k_2-1)\dotsm(k_z-1)$.
\end{lemma}
\begin{proof}
Write $s=s(\mathcal A,\mathcal B,\mathcal C)$. By \cref{def:Y1},
\[
s=\sum_{K\in\mathcal A}(2w_K+w_{2K}-l_K w_K)
+\sum_{H\in\mathcal B}(w_{4H}+w_H)
+\sum_{H\in\mathcal C}w_{3H}.
\]
A key fact for simplifying this formula is that if $K=\alpha I\beta J$, then
\[
w_K=w_{\alpha i_1\beta}=w_{K(i_1)}w_{I\!J}'
\quad\text{and}\quad
w_{uK}=w_{u\alpha I\beta J}=u\cdotp (i_1-1)w_{1\theta}w_{I\!J}'.
\]
On the other hand, for any $H\in\mathcal A(k)$, we can write $H=\alpha' I\beta' J'$, where $\alpha'$ and $\beta'$ are compositions such that $\rho(k\alpha'\beta')=\rho(\alpha\beta)$. Then
\[
w_{vH}=w_{v\alpha' I\beta' J}
=vw_{1\alpha'\beta'}\cdot (i_1-1)w_{I\!J}'
=v\cdot\frac{w_{1\theta}}{k-1}\cdot (i_1-1)w_{I\!J}'.
\]
Therefore, 
\[
\frac{s}{w_{I\!J}'}=\sum_{K\in\mathcal A}
\brk1{(2-l_K)w_{K(i_1)}+2(i_1-1)w_{1\theta}}
+\sum_{H\in\mathcal B}
\brk1{4(i_1-1)w_{1\theta}+w_{H(i_1)}}+
\sum_{H\in\mathcal C}\frac{3}{2}(i_1-1)w_{1\theta}.
\]
Since $w_{1\theta}$ is independent of $K$, we can transform the formula above to the desired one.
\end{proof}

By \cref{lem:computation.Y1}, the goal $s\ge 0$ in Step 2 reduces to $c(\mathcal A,\mathcal B,\mathcal C)\ge 0$ for each set $\mathcal A=\mathcal F_k(I,J;P)$. 

We handle the family~$\mathcal F_3$, $\dots$, $\mathcal F_6$ in \cref{lem3,lem4,lem5,lem6}, respectively. For compactness, we will omit the set braces for $\mathcal A$, $\mathcal B$ and $\mathcal C$ when the set is a singleton. 

\begin{lemma}\label{lem3}
We have the following.
\begin{enumerate}
\item
For any composition $I2J\in\mathcal E_{k(k+2)}$ with $\abs{J}=b+4-k$ for some $k\in[4]$, 
\[
c(I2J,\,I\!J,\,\emptyset)
=\brk1{9-4\chi(k\in\{1,4\})-5\chi(k\in\{2,3\})}i_1-6.
\]
\item
For any composition $I3J\in\mathcal E_{k(k+3)}$ with $\abs{J}=b+3-k$ for some $k\in[3]$, 
\[
c(I3J,\,\emptyset,\,I\!J)=3i_1-7.
\]
\item
For any composition $I22J\in\mathcal E_{k(k+2)(k+4)}$ with $\abs{J}=b+2-k$ for some $k\in[2]$, 
\[
c(I22J,\,I2J,\,\emptyset)=3(i_1-2).
\]
\end{enumerate}
\end{lemma}
\begin{proof}
We show them by \cref{lem:computation.Y1}. 
\begin{enumerate}
\item
By definition, $l_{I2J}=4\chi(k\in\{1,4\})+5\chi(k\in\{2,3\})$.
Then 
\[
c(I2J,\,I\!J,\,\emptyset)=(2-l_{I2J})w_{i_1}+w_{i_1}+(2+4)(i_1-1),
\]
which simplifies to the desired formula.
\item
In this case, we have
$c(I3J, \,\emptyset, \,I\!J)=(2-4)w_{i_13}+(2+3/2)(i_1-1)w_{13}=3i_1-7$.
\item
In this case, we have
$c(I22J, \,I2J, \,\emptyset)=(2-6)w_{i_1}+w_{i_1}+(2+4)(i_1-1)=3(i_1-2)$.
\end{enumerate}
This completes the proof.
\end{proof}

\begin{lemma}\label{lem4}
There exists a decomposition \cref{decomp:Y1.Fk} for $k=4$, such that for each $\mathcal A=\mathcal A(I,J;\theta)$, there exist sets $\mathcal B$ and $\mathcal C$ with $c(\mathcal A, \mathcal B, \mathcal C)>0$. Moreover, all the sets $\mathcal B$ and $\mathcal C$ are pairwise disjoint, and
\begin{align*}
\bigcup \mathcal B
&=\{I\!J\in\mathcal W_{2b+4}\colon\abs{J}=b+2,\ j_1\ge 3\}, 
\quad\text{and}\\
\bigcup \mathcal C
&=\{I\!J\in\mathcal W_{2b+3}\colon\abs{J}=b+2,\ j_1\ge 3,\ i_1\ge 3\}.
\end{align*}
\end{lemma}
\begin{proof}
We can decompose $\mathcal F_4$ into pairwise disjoint sets $\mathcal A_0,\dots,\mathcal A_{d+1}$, where $d=\floor{(b+1)/2}$, 
\begin{align*}
\mathcal A_0&=\{I3J\colon i_1\ge 3\},\\
\mathcal A_k&=\{2^kI3J,\ I32^kJ\colon i_1\ge 3\},\quad\text{for $1\le k\le d$},\quad\text{and}\\
\mathcal A_{d+1}&=\mathcal E_{24}\backslash \cup_{k=1}^d\{I32^kJ\colon i_1\ge 3\}.
\end{align*}
Here we omitted the common conditions $\mathcal A_i\subset\mathcal F_4$, $\abs{J}=b+2$, and $j_1\ge 3$ for brevity. 
\begin{itemize}
\item
For $\mathcal A_0$, we have $c(I3J,\,\emptyset,\,I\!J)=3i_1-7>0$ by \cref{lem3}.
\item
For $\mathcal A_k$ with $k\in[d]$, by \cref{lem:computation.Y1},
\[
c\brk1{\{2^kI3J,\,I32^kJ\},\,\{I32^{k-1}J\},\,\emptyset}
=-2w_{2i_13}-3w_{i_13}+w_{i_13}+(2\cdotp 2+4)(i_1-1)w_{13}=4(i_1-2)>0.
\]
\item
For $\mathcal A_{d+1}$, we have $c(I2J,\,I\!J,\,\emptyset)=4x-6>0$ by \cref{lem3}.
\end{itemize}
It is routine to check that the sets $\mathcal B$ and $\mathcal C$ are disjoint, and the unions have the desired form. 
\end{proof}

\begin{lemma}\label{lem5}
There exists a decomposition \cref{decomp:Y1.Fk} for $k=5$, such that for each $\mathcal A=\mathcal A(I,J;\theta)$, there exist sets $\mathcal B$ and $\mathcal C$ with $c(\mathcal A, \mathcal B, \mathcal C)\ge 0$. Moreover, all the sets $\mathcal B$ and $\mathcal C$ are pairwise disjoint, and
\begin{align*}
\bigcup \mathcal B
&=\{I\!J\in\mathcal W_{2b+4}\colon\abs{J}=b+1\}
\quad\text{and}\\
\bigcup \mathcal C
&=\{I\!J\in\mathcal W_{2b+3}\colon\abs{J}=b+1,\ i_1\ge 3,\ i_{-1}\ge 3\}.
\end{align*}
\end{lemma}
\begin{proof}
We can decompose $\mathcal E_{25}\cup\mathcal E_{35}$ into pairwise disjoint sets $\mathcal A_0,\dots,\mathcal A_{b}$, where
\begin{align*}
\mathcal A_0&=\{I3J\colon i_1\ge 3,\ i_{-1}\ge 3\},
&
\mathcal A_1&=\{I2J\colon i_{-1}\ge 4\},
\\
\mathcal A_2&=\{2I23J,\ 2I32J\colon I, J\},
&
\mathcal A_k&=\{2Ik3J,\ kI23J,\ kI32J\colon I, J\},\quad\text{for $3\le k\le b$}.
\end{align*}
Here we omitted the condition that the composition pair $(I,J)$ runs over all possibilities such that $\mathcal A_i\subset\mathcal F_5$ and $\abs{J}=b+1$. We deal with $\mathcal F_5=\mathcal E_{25}\cup\mathcal E_{35}\cup\mathcal E_{135}$ in $5$ classes.
\begin{enumerate}
\item
For $\mathcal A_0$, we have $c(I3J,\,\emptyset,\,I\!J)=3i_1-7>0$ by \cref{lem3}.
\item
For $\mathcal A_1$, we have $c(I2J,\,I\!J,\,\emptyset)=4i_1-6>0$ by \cref{lem3}.
\item
For $\mathcal A_2$, by \cref{lem:computation.Y1}, 
\[
c(\{2I23J,\ 2I32J\},\,2I3J,\,\emptyset)
=-2w_{2i_13}-3w_{2i_13}+w_{2i_13}+(2\cdotp 2+4)(i_1-1)w_{13}=0.
\]
\item
For $\mathcal A_k$ with $3\le k\le b$, by \cref{lem:computation.Y1}, 
\begin{align*}
c(\{2Ik3J,\ kI23J,\ kI32J\},\,kI3J,\,\emptyset)
&=-2w_{2i_1k3}-2w_{ki_13}-3w_{ki_13}+w_{ki_13}+(2\cdotp 3+4)(i_1-1)w_{1k3}\\
&=4(i_1-1)(k-3)\ge 0.
\end{align*}
\item
For $\mathcal E_{135}$, we have $c(I22J,\,I2J,\,\emptyset)=3(i_1-2)\ge 0$ by \cref{lem3}.
\end{enumerate}
It is routine to check that the sets $\mathcal B$ and $\mathcal C$ are disjoint, and the unions have the desired form. 
\end{proof}

\begin{lemma}\label{lem6}
There exists a decomposition \cref{decomp:Y1.Fk} for $k=6$, such that for each $\mathcal A=\mathcal A(I,J;\theta)$, there exist sets $\mathcal B$ and $\mathcal C$ with $c(\mathcal A, \mathcal B, \mathcal C)\ge 0$. Moreover, all the sets $\mathcal B$ and $\mathcal C$ are pairwise disjoint, and
\[
\bigcup \mathcal B
=\{I\!J\in\mathcal W_{2b+4}\colon\abs{J}=b\}
\quad\text{and}\quad
\bigcup \mathcal C
=\{I\!J\in\mathcal W_{2b+3}\colon\abs{J}=b\}.
\]
\end{lemma}
\begin{proof}
We decompose each set $\mathcal E_S$ in $\mathcal F_6$ as $\mathcal E_S^{1}\sqcup\mathcal E_S^{2}\sqcup\mathcal E_S^{3}$ in the way demonstrated in \cref{tab:decomp6}. Here we omitted the condition that the composition pair $(I,J)$ runs over all possibilities such that $\mathcal E_S^i\subseteq \mathcal F_6$ and $\abs{J}=b$.
\begin{table}[htbp]
\centering
\caption{The decompositions $\mathcal E_S=\mathcal E_S^{1}\sqcup\mathcal E_S^{2}\sqcup\mathcal E_S^{3}$.}
\label{tab:decomp6}
\begin{tabular}{@{} CLLLLL @{}}
\toprule
\addlinespace
i
& \mathcal E_{36}^{i} 
& \mathcal E_{46}^{i} 
& \mathcal E_{136}^{i} 
& \mathcal E_{146}^{i} 
& \mathcal E_{246}^{i}  \\ 
\addlinespace
\midrule
\addlinespace
1
& \underline{\{2Ik3J\colon k\ge 4\}}
& \underline{\{3Ik2J\colon k\ge 4\}}
& \underline{\{kI23J\colon k\ge 4\}}
& \underline{\{kI32J\colon k\ge 4\}}
& \{2I22J\}
\\[3pt]
\addlinespace
2
& \uwave{\{2I33J\}}
& \uuline{\{2Ik2J\colon k\ge 4\}}
& \uwave{\{3I23J\}}
& \uwave{\{3I32J\}}
& \uuline{\{kI22J\colon k\ge 4\}}
\\[3pt]
\addlinespace
3
& \{I3J\colon i_1,\,i_{-1}\ge 3\}
& \{I2J\colon i_1,\,i_{-1}\ge 4\}
& \dashuline{\{2I23J\}}
& \dashuline{\{2I32J\}}
& \dashuline{\{3I22J\}}
\\
\bottomrule
\end{tabular}
\end{table}
We deal with $\mathcal F_6$ in $7$ classes. In \cref{tab:decomp6}, we mark $4$ of the classes by underlines, double underlines, waves, and dashes, respectively, and leave each of the remaining $3$ classes $\mathcal E_S^i$ unmarked. In each case, we fix the composition pair $(I,J)$ and write $c=c(\mathcal A,\mathcal B,\mathcal C)$.
\begin{enumerate}
\item
Let $k\ge 4$. For $\mathcal A=\{2Ik3J,\ 3Ik2J,\ kI23J,\ kI32J\}$, we take \[
\mathcal B=\{3IkJ,\ kI3J\}
\quad\text{and}\quad
\mathcal C=\{2IkJ,\ kI2J\}.
\]
By \cref{lem:computation.Y1},
\begin{align*}
c&=
-2w_{2i_1k3}-2w_{3i_1k}-3w_{ki_13}-3w_{ki_13}+w_{3i_1k}+w_{ki_13}+\brk3{2\cdotp 4+4\cdotp 2+\frac{3}{2}\cdotp 2}(i_1-1)w_{1k3}\\
&=(17k-27)(i_1-1)>0.
\end{align*}
\item
Let $k\ge 4$. For $\mathcal A=\{2Ik2J,\,kI22J\}$, we take $\mathcal B=\{2IkJ,\,kI2J\}$ and $\mathcal C=\emptyset$. By \cref{lem:computation.Y1}, 
\[
c=-2w_{2i_1k}-4w_{ki_1}+w_{2i_1k}+w_{ki_1}+(2\cdotp 2+4\cdotp 2)(i_1-1)w_{1k}
=(7k-10)(i_1-1)>0.
\]
\item
For $\mathcal A=\{2I33J,\ 3I23J,\ 3I32J\}$, we take $\mathcal B=\{3I3J\}$ and $\mathcal C=\{2I3J,\ 3I2J\}$. By \cref{lem:computation.Y1},
\[
c=-2w_{2i_133}-3w_{3i_13}-3w_{3i_13}+w_{3i_13}+\brk3{2\cdotp 3+4+\frac{3}{2}\cdotp2}(i_1-1)w_{133}
=6(i_1-1)>0.
\]
\item
For $\mathcal A=\{2I23J,\ 2I32J,\ 3I22J\}$, we take $\mathcal B=\{2I3J,\ 3I2J\}$ and $\mathcal C=\{2I2J\}$. By \cref{lem:computation.Y1},
\[
c=-3w_{2i_13}-3w_{2i_13}-4w_{3i_1}+w_{2i_13}+w_{3i_1}+\brk3{2\cdotp 3+4\cdotp 2+\frac{3}{2}}(i_1-1)w_{13}
=2(i_1-1)>0.
\]
\item
For $\mathcal A\in\mathcal E_{36}^3$, we have $c(I3J,\,\emptyset,\,I\!J)=3i_1-7>0$ by \cref{lem3}.
\item
For $\mathcal A\in\mathcal E_{46}^{3}$, we have $c(I2J,\,I\!J,\emptyset)=5i_1-6>0$ by \cref{lem3}.
\item
For $\mathcal A\in\mathcal E_{246}^{1}$, we have $c(2I22J,\,2I2J,\,\emptyset)=0$ by \cref{lem3}.
\end{enumerate}
It is routine to check that the sets $\mathcal B$ and $\mathcal C$ are disjoint, and the unions have the desired form. 
\end{proof}

Now we can set up the $e$-positivity of $Y_1$.
\begin{lemma}\label{lem:e+:trinacria1.Y1}
The symmetric function $Y_1$ in \cref{prop:X.3sun:b+2.b.2} is $e$-positive.
\end{lemma}
\begin{proof}
\Cref{lem3,lem4,lem5,lem6} completes the first two steps of our strategy at the beginning of \cref{sec:Y1}. For $3\le k\le 6$, denote by $\mathcal B_k$ (resp., $\mathcal C_k$) the union $\bigcup\mathcal B$ (resp., $\bigcup\mathcal C$) for $\mathcal F_k$. By \cref{lem3,lem4,lem5,lem6},
\begin{align*}
\mathcal B_3
&=\{I\!J\in\mathcal W_{2b+4}\colon\abs{J}=b+3,\ j_1\ge 4\},
&
\mathcal C_3
&=\emptyset,
\\
\mathcal B_4
&=\{I\!J\in\mathcal W_{2b+4}\colon\abs{J}=b+2,\ j_1\ge 3\},
&
\mathcal C_4
&=\{I\!J\in\mathcal W_{2b+3}\colon\abs{J}=b+2,\ j_1\ge 3,\ i_1\ge 3\},
\\
\mathcal B_5
&=\{I\!J\in\mathcal W_{2b+4}\colon\abs{J}=b+1\},
&
\mathcal C_5
&=\{I\!J\in\mathcal W_{2b+3}\colon\abs{J}=b+1,\ i_1\ge 3,\ i_{-1}\ge 3\},
\\
\mathcal B_6
&=\{I\!J\in\mathcal W_{2b+4}\colon\abs{J}=b\},
&
\mathcal C_6
&=\{I\!J\in\mathcal W_{2b+3}\colon\abs{J}=b\}.
\end{align*}
It is clear that these sets $\mathcal B_k$ and $\mathcal C_k$ are pairwise disjoint. This completes the proof.
\end{proof}

\subsection{$Y_0$ is $e$-positive}\label{sec:Y0}
In this section, we write $\mathcal W=\mathcal W_{2b+7}$, and let $L_K$ be the set of numbers $l\in[3]$ such that $K$ has a suffix of size $b+l$. Then $L_K\in\brk[c]1{\{1\},\, \{2\},\, \{3\}, \,\{1, 3\}}$. Let 
\[
\mathcal K_0=\{K\in\mathcal W\colon L_K\cap\{2, 3\}\ne \emptyset\}.
\]
If $K\not\in\mathcal K_0$, then \cref{def:f} gives 
\[
f(K)\ge 2r_{k_1}-r_{k_1}r_{j_1}\chi(1\in L_K)
\ge r_{k_1}(2-r_{j_1})\ge 0.
\]
When $K\in\mathcal K_0$, 
we can factorize $K$ uniquely as 
\begin{equation}\creflabel[def]{def:K=IJ:Y0}
K=I\!J, \quad\text{where $J\vDash b+l_K$ and $l_K\in\{2, 3\}$}.
\end{equation} 
It follows that $\abs{I}=b+7-l_K\ge b+4$. Therefore, if $k_1\le 3$, then $k_2=i_2$. Then \cref{def:f} reduces to
\begin{equation}\label{def:f.l=23}
f(K)=2r_{i_1}+\delta_{i_1} r_{i_2}-l_K r_{i_1}r_{j_1}-r_{i_1} r_{j_2}\chi(1\in L_K), 
\end{equation}
where the condition $1\in L_K$ is equivalent to ``$l_K=3$ and $j_1=2$,'' and 
\[
\delta_x=4\chi(x=2)+\frac{3}{2}\chi(x=3).
\]
Define a map $\varphi\colon\mathcal K_0\to\mathcal W$ by $\varphi(K)=J\!I$, where $I$ and $J$ are defined by \cref{def:K=IJ:Y0}. Then $\varphi(K)$ and~$K$ have the same underlying partition, and $\varphi$ is injective. Let 
\[
F(K)=f(K)+f\brk1{\varphi(K)}.
\]
The \emph{preimage} $\varphi^{-1}(H)$ is the set of compositions $K$ such that $\varphi(K)=H$. Let \[
\mathcal K^-=\{K\in\mathcal W\colon f(K)<0\}, 
\quad
\mathcal K^+=\{K\in\mathcal W\colon f(K)\ge 0\}, 
\quad\text{and}\quad
\mathcal H=\{K\in\mathcal K^+\colon\varphi^{-1}(K)=\emptyset\}.
\] 

Our proof for the $e$-positivity of $Y_0$ adopts the following \emph{progressive repair method}. First, we give a coarse solution for each composition $K\in\mathcal K^-$: possibly $F(K)\ge 0$. Second, we provide a modified solution for those with $F(K)<0$: possibly $F(K)+f(U)\ge 0$ for some $U\in\mathcal H$. Thirdly, for the remaining compositions $K$, we modify our solution again: possibly $F(K)+F(K')\ge 0$ for another composition $K'\in\mathcal K^-$. At last, for each composition $K$ that is left unsolved, we will find an extra composition $V\in\mathcal H$ and show that $F(K)+F(K')+f(V)\ge 0$. Briefly speaking, we shall construct an involution $\xi$ on $\mathcal K^-$ and show the following:
\begin{description}
\item[Step 1. Solving coarsely]
If $\xi(K)=K$, then $F(K)\ge 0$ or $F(K)+f(U)>0$ for some $U\in\mathcal H$. 
\item[Step 2. Repairing]
If $\xi(K)=K'\ne K$, then either $F(K)+F(K')>0$, or $F(K)+F(K')+f(V)\ge 0$ for some $V\in\mathcal H$.
\item[Step 3. Distinctness verification] 
All the compositions $U$ and $V$ are distinct.
\end{description}
Since $\varphi$ is injective, proving these statements yields the $e$-positivity of $Y_0$.

Now let us start the proof. Note that $\mathcal K^-\subseteq \mathcal K_0$. We characterize $\mathcal K^-$ first.

\begin{proposition}\label{prop:f<0.Y0}
We have $\mathcal K^-=\mathcal A\sqcup \mathcal B\sqcup \mathcal C\sqcup \mathcal D\sqcup \mathcal E$, where
\begin{align*}
\mathcal A
&=\{K=I\!J\in\mathcal W\colon J\vDash b+3,\ (i_1, j_1)=(2, 3),\ i_2\ge 6\},\\
\mathcal B
&=\{K=I\!J\in\mathcal W\colon J\vDash b+2,\ (i_1, j_1)=(3, 2)\},\\
\mathcal C
&=\{K=I\!J\in\mathcal W\colon J\vDash b+3,\ i_1=3,\ j_1\le 3i_2-3\},\\
\mathcal D
&=\{K=I\!J\in\mathcal W\colon J\vDash b+3,\ (i_1, j_1)=(2, 2)\},\\
\mathcal E
&=\{K=I\!J\in\mathcal W\colon J\vDash b+l_K,\ l_K\in\{2,3\},\ i_1\ge 4\}.
\end{align*}
\end{proposition}
\begin{proof}
For $K\in\mathcal K^-$, we will compute $f(K)$ by using \cref{def:f.l=23}, and determine its sign using the uniform bounds $1<r_i\le 2$ for all $i$. If $i_1\ge 4$, then
\[
f(K)
=2r_{i_1}
-l_K r_{i_1}r_{j_1}
-r_{i_1}r_{j_2}\chi(1\in L_K)
\le 2r_{i_1}(1-r_{j_1})<0.
\]
Below we can suppose that $i_1\le 3$. We proceed according to the value of $l_K$. Suppose that $l_K=2$. 
\begin{itemize}
\item
If $i_1=2$, then $f(K)=4+4r_{i_2}-4r_{j_1}>4+4-4\cdotp 2=0$.
\item
If $i_1=3$, then $f(K)=3+r_3r_{i_2}-3r_{j_1}-r_3r_{j_2}\chi(j_1=2)$.
In this case, 
\vskip 4pt
\begin{itemize}
\item
if $j_1=2$, then $f(K)=r_3(r_{i_2}-r_{j_2})-3<r_3(2-1)-3<0$; 
\item
if $j_1\ge 3$, then $f(K)>3+r_3-3r_3=0$.
\end{itemize}
\end{itemize}
Now suppose that $l_K=3$.
\vskip 4pt
\begin{itemize}
\item
If $i_1=2$, then $f(K)=4+4r_{i_2}-6r_{j_1}-2r_{j_2}\chi(1\in L_K)$. In this case,
\vskip 4pt
\begin{itemize}
\item
if $j_1=2$, then 
$f(K)=4r_{i_2}-8-2r_{j_2}<4\cdotp 2-8-2<0$;
\item
if $j_1=3$, then $f(K)=4r_{i_2}-5$, which is negative if and only if $i_2\ge 6$;
\item
if $j_1\ge 4$, then $f(K)=4+4r_{i_2}-6r_{j_1}>4+4-6r_4=0$.
\end{itemize}
\item
If $i_1=3$, then $f(K)=3+r_3r_{i_2}-3r_3r_{j_1}-r_3r_{j_2}\chi(j_1=2)$. In this case,
\vskip 4pt
\begin{itemize}
\item
if $j_1=2$, then $f(K)=r_3(r_{i_2}-r_{j_2}-4)<r_3(2-1-4)<0$;
\item
if $j_1\ge 3$, then $f(K)=r_3(2+r_{i_2}-3r_{j_1})$, which is negative if and only if $j_1\le 3i_2-3$.
\end{itemize}
\end{itemize}
In summary, we obtain the desired characterization for $\mathcal K^-$.
\end{proof}

\begin{lemma}\label{lem:f(varphi)>=0}
We have $\varphi(\mathcal K^-)\subseteq\mathcal K^+$.
\end{lemma}
\begin{proof}
Let $K=I\!J$, where $I$ and $J$ are defined by \cref{def:K=IJ:Y0}. Let $H=\varphi(K)=J\!I$, where $\abs{I}=b+7-l_K\in\{b+4, \,b+5\}$. Assume to the contrary that $f(H)<0$. Then $H\in\mathcal K_0$, and there exists $l_H\in\{2, 3\}$ such that 
\begin{equation}\label{pf.fH>0:i1}
i_1=\abs{I}-(b+l_H)=7-l_K-l_H.
\end{equation}

If $l_K=3$, then \cref{pf.fH>0:i1} implies $l_H=i_1=2$. On one hand, since $f(K)<0$ and $(l_K, i_1)=(3, 2)$, \cref{prop:f<0.Y0} gives either $j_1=2$ or ``$j_1=3$ and $i_2\ge 6$''. On the other hand, since $f(H)<0$ and $l_H=2$, \cref{prop:f<0.Y0} gives either $j_1\ge 4$ or ``$j_1=3$ and $i_2=2$'', a contradiction. 

Otherwise $l_K=2$. In this case, if $l_H=3$, then \cref{pf.fH>0:i1} implies $i_1=2$. Then \cref{prop:f<0.Y0} gives $f(K)\ge 0$, a contradiction. Thus $l_H=2$, and \cref{pf.fH>0:i1} implies $i_1=3$. Then \cref{prop:f<0.Y0} gives $j_1=2$. On the other hand, since $f(H)<0$ and $l_H=2$, \cref{prop:f<0.Y0} gives $j_1=h_1\ge 3$, a contradiction. 

In conclusion, the assumption $f(H)<0$ is false and we obtain a proof.
\end{proof}

\cref{lem:F} gives a formula for $F(K)$.

\begin{lemma}\label{lem:F}
Let $K=I\!J\in\mathcal K_0$ be defined by \cref{def:K=IJ:Y0}. If $l_K+i_1\ge 5$, then
\[
F(K)
=\sum_{h\in\{i,\, j\}}(2r_{h_1}+\delta_{h_1}r_{h_2})
-l_K r_{i_1}r_{j_1}-r_{i_1}r_{j_2}\chi(1\in L_K)
-(7-l_K-i_1)r_{j_1}r_{i_2}\chi(i_1\le 4).
\]
\end{lemma}
\begin{proof}
Since $\abs{I}=b+7-l_K\in\{b+4,\,b+5\}$ and since $l_K+i_1\ge 5$, we can infer from \cref{def:f.l=23} that
\[
f(\varphi(K))
=2r_{j_1}+\delta_{j_1}r_{j_2}-(7-l_K-i_1)r_{j_1}r_{i_2}\chi(i_1\le 4).
\]
Adding it with \cref{def:f.l=23} yields the desired formula.
\end{proof}

We then focus on the set $\mathcal S=\{K\in\mathcal K_0\colon f(K)<0,\ F(K)<0\}=\mathcal S_2\sqcup\mathcal S_3$, where 
\[
\mathcal S_i=\{K\colon\mathcal S\colon l_K=i\}.
\]
We characterize $\mathcal S_2$ and $\mathcal S_3$ in \cref{prop:F.l=2,prop:F.l=3}, respectively.

\begin{proposition}\label{prop:F.l=2}
Let $K=I\!J\in\mathcal K_0$ be defined by \cref{def:K=IJ:Y0}. Then $K\in\mathcal S_2$ if and only if one of the following happens:
\begin{enumerate}
\item
$(i_1, i_2)=(4, 2)$ and $j_1\ge 4$, 
\item
$(i_1, i_2)=(4, 3)$ and $j_1\in\{4, 5\}$.
\end{enumerate}
Moreover, in either of these cases, we have $F(K)=(8+3r_{i_2}r_{j_1}-2r_{j_1})/3$.
\end{proposition}
\begin{proof}
By \cref{prop:f<0.Y0}, we find either $(i_1, j_1)=(3, 2)$ or $i_1\ge 4$.
Then by \cref{lem:F}, 
\begin{equation}\label{pf:F.l=2}
F(K)=\sum_{h\in\{i,\, j\}}(2r_{h_1}+\delta_{h_1}r_{h_2})
-2r_{i_1}r_{j_1}-(5-i_1)r_{j_1}r_{i_2}\chi(i_1\le 4).
\end{equation}
If $(i_1, j_1)=(3, 2)$, then \cref{pf:F.l=2} gives
\begin{equation}\label{F.232}
F(K)=1+4r_{j_2}-\frac{5}{2}r_{i_2}>1+4-\frac{5}{2}\cdotp 2=0.
\end{equation}
It is routine to check that $F>0$ for when $i_1\ge 5$. Let $i_1=4$. Then
\begin{equation}\label{pf:F.34}
F(K)=\frac{8}{3}+\delta_{j_1}r_{j_2}-\frac{2}{3}r_{j_1}-r_{j_1}r_{i_2}.
\end{equation}
In this case, if $j_1=2$, then \cref{pf:F.34} gives
\[
F(K)=\frac{8}{3}+4r_{j_2}-\frac{2}{3}\cdotp 2-2r_{i_2}
>\frac{4}{3}+4-2\cdotp 2>0;
\]
if $j_1=3$, then \cref{pf:F.34} gives
\begin{equation}\label{F.243}
F(K)=\frac{5}{3}+\frac{3}{2}(r_{j_2}-r_{i_2})
>\frac{5}{3}+\frac{3}{2}(1-2)>0.
\end{equation}
Suppose that $j_1\ge 4$. Then $F(K)$ becomes the desired form. If $i_2=2$, then
\[
F(K)=\frac{8}{3}(1-r_{j_1})<0.
\]
If $i_2=3$, then 
\[
F(K)=\frac{8}{3}-\frac{13}{6}r_{j_1}, 
\]
which is negative if and only if $j_1\le 5$. If $i_2\ge 4$, then 
\[
F(K)\ge \frac{8}{3}-\brk2{\frac{2}{3}+\frac{4}{3}}\cdotp \frac{4}{3}=0.
\]
In summary, we obtain the desired statement.
\end{proof}

\begin{proposition}\label{prop:F.l=3}
Let $K=I\!J\in\mathcal K_0$ be defined by \cref{def:K=IJ:Y0}. Then $K\in\mathcal S_3$ if and only if one of the following happens:
\begin{enumerate}
\item
$(i_1, j_1)=(3, 2)$, $5i_2-j_2-3\le 0$, and $F(K)=(5r_{j_2}-r_{i_2}-4)/2$, 
\item
$(i_1, j_1)=(3, 4)$, $i_2\ge 3$, and $F(K)=(r_{i_2}-2)/6$, 
\item
$(i_1, j_1)=(2, 2)$, $j_2\ge 3$, and $F(K)=2(r_{j_2}-2)$.
\end{enumerate}
\end{proposition}
\begin{proof}
By \cref{lem:F}, 
\begin{equation}\label{F.l=3}
F(K)=\sum_{h\in\{i,\, j\}}(2r_{h_1}+\delta_{h_1}r_{h_2})
-3r_{i_1}r_{j_1}-r_{i_1}r_{j_2}\chi(j_1=2)
-(4-i_1)r_{j_1}r_{i_2}\chi(i_1\le 3).
\end{equation}
We proceed according to the value of $i_1$.
\begin{itemize}
\item
If $i_1\ge 4$, then \cref{F.l=3} reduces to
\[
F(K)=2r_{i_1}+2r_{j_1}+\delta_{j_1}r_{j_2}
-3r_{i_1}r_{j_1}-r_{i_1}r_{j_2}\chi(j_1=2).
\]
It is routine to check that $F\ge 0$ according to the value of $j_1$, if one notices $1<r_i\le 2$ for all $i$.
\item
If $i_1=3$, then \cref{F.l=3} reduces to
\begin{equation}\label{pf:F.l=i1=3}
F(K)=3+\frac{3}{2}r_{i_2}+\delta_{j_1}r_{j_2}
-\frac{5}{2}r_{j_1}-\frac{3}{2}r_{j_2}\chi(j_1=2)-r_{j_1}r_{i_2}.
\end{equation}
It is routine to check $F(K)>0$ for when $j_1\ge 5$ or $j_1=3$.
If $j_1=4$, then $f(K)<0$ implies $i_2\ge 3$ by \cref{prop:f<0.Y0}, and \cref{pf:F.l=i1=3} reduces to the desired negative expression of $F(K)$. If $j_1=2$, then $f(K)<0$ by \cref{prop:f<0.Y0}. In this case, \cref{pf:F.l=i1=3} reduces to the desired expression of $F(K)$. It is elementary to deduce that $F(K)<0$ if and only if $5i_2-j_2-3\le 0$.
\item
If $i_1=2$, then $f(K)<0$ implies $j_1\le 3$ by \cref{prop:f<0.Y0}, and \cref{F.l=3} reduces to
\begin{equation}\label{F.22}
F(K)=4+4r_{i_2}+\delta_{j_1}r_{j_2}
-4r_{j_1}-2r_{j_2}\chi(j_1=2)-2r_{j_1}r_{i_2}.
\end{equation}
In this case, if $j_1=3$, then \cref{F.22} gives
\[
F(K)=\frac{5}{2}r_{j_2}-2>\frac{5}{2}-2>0.
\]
If $j_1=2$, then \cref{F.22} reduces to $F(K)=2(r_{j_2}-2)$, which is negative if and only if $j_2\ge 3$.
\end{itemize}
In summary, we obtain the desired characterization for $F(K)<0$.
\end{proof}

Following Step 2 of our strategy, we deal with compositions $K\in\mathcal S_2$.

\begin{lemma}\label{lem:24>=4.i2=23}
Let $K=I\!J\in\mathcal K_0$ be defined by \cref{def:K=IJ:Y0} with $l_K=2$.
If $f(K)<0$, then $K=4ST$ for some compositions $S$ and $T$ such that $(\abs{S}, \abs{T})=(b+1, \,b+2)$, and $s_1\le 3<4\le t_1$. Moreover, for the composition $M=S4T$, 
\[
F(K)+f(M)
=\frac{8}{3}+2r_{s_1}+\delta_{s_1}r_{s_2}
-\frac{2}{3}r_{t_1}-3r_{s_1}r_{t_1}
>0.
\]
\end{lemma}
\begin{proof}
By \cref{prop:F.l=2}, the composition $K$ has the desired form $K=4ST$. Together with \cref{def:f.l=23}, we can calculate
\begin{align*}
F(K)+f(M)
&=\brk3{\frac{8}{3}-\frac{2}{3}r_{t_1}-r_{s_1}r_{t_1}}
+(2r_{s_1}+\delta_{s_1}r_{s_2}-2r_{s_1}r_{t_1})\\
&>\frac{8}{3}-\frac{2}{3}\cdotp\frac{4}{3}
+2r_{s_1}+\delta_{s_1}-3r_{s_1}\cdotp\frac{4}{3}
=\frac{16}{9}+\delta_{s_1}-2r_{s_1}\\
&\ge\frac{16}{9}+\min\brk3{4-2\cdotp 2, \
\frac{3}{2}-2\cdotp \frac{3}{2}}>0.
\end{align*}
This completes the proof.
\end{proof}

\cref{lem:332,lem:334,lem:322} deal with the three kinds of compositions $K\in\mathcal S_3$, respectively.

\begin{lemma}\label{lem:332}
Let $\mathcal M=\{K=3P2Q\in\mathcal W_{2b+7}\colon\abs{P}=\abs{Q}=b+1\}$. Then for any $K\in\mathcal M$, 
\[
F(K)+F\brk1{\Phi(K)}>0,
\]
where $\Phi$ is the involution on $\mathcal M$ defined by $\Phi(3P2Q)=3Q2P$.
\end{lemma}
\begin{proof}
Let $K=3P2Q\in\mathcal M$. Then $\Phi(K)=3Q2P$. By \cref{prop:F.l=3} and by symmetry, we obtain
\[
F(K)+F\brk1{\Phi(K)}
=\frac{5r_{q_1}-r_{p_1}-4}{2}+\frac{5r_{p_1}-r_{q_1}-4}{2}
=2r_{p_1}+2r_{q_1}-4>2+2-4=0.
\]
This completes the proof.
\end{proof}

\begin{lemma}\label{lem:334}
Let $K=I\!J\in\mathcal K_0$ be defined by \cref{def:K=IJ:Y0} with $(l_K, i_1, j_1)=(3, 3, 4)$ and $i_2\ge 3$. Then $K=3P4Q$ for some compositions $P$ and $Q$ such that $(\abs{P}, \abs{Q})=(b+1, \,b-1)$, and 
\[
F(K)+F(K')>0,\quad\text{where $K'=4P3Q$}.
\] 
\end{lemma}
\begin{proof}
Since $(l_{K'}, i_1, j_1)=(2, 4, 3)$, we find $f(K')<0$ by \cref{prop:f<0.Y0}. By \cref{prop:F.l=3,F.243}, 
\[
F(K)+F(K')
=\frac{r_{p_1}-2}{6}+\brk3{\frac{5}{3}+\frac{3}{2}(r_{q_1}-r_{p_1})}
=\frac{4}{3}-\frac{4}{3}r_{p_1}+\frac{3}{2}r_{q_1}
>\frac{4}{3}-\frac{4}{3}\cdotp 2+\frac{3}{2}>0.
\]
This completes the proof.
\end{proof}

Recall that $\varphi$ is injective. When $\varphi^{-1}(H)=\{K\}$ is a singleton, we write $\varphi^{-1}(H)=K$.

\begin{lemma}\label{lem:322}
Let $K=I\!J\in\mathcal K_0$ be defined by \cref{def:K=IJ:Y0} with $(l_K, i_1, j_1)=(3, 2, 2)$ and $j_2\ge 3$. Then $K=2P2j_2Q\in\mathcal W_{2b+7}$ with $(\abs{P}, \abs{j_2Q})=(b+2, \,b+1)$. Let $R=22Qj_2P$. Then $F(K)+f(R)>0$. Moreover, we have the following:
\begin{itemize}
\item
if $j_2=3$, then the composition $K'=\varphi^{-1}(R)$ exists, and 
\[
F(K)+F(K')>0;
\]
\item
if $j_2=4$, then $\varphi^{-1}(R)=\emptyset$, and for the composition $K''=422Q\!P$,
\[
F(K)+f(R)+F(K'')\ge 0.
\]
\end{itemize}
\end{lemma}
\begin{proof}
Note that $\varphi^{-1}(R)=\emptyset$ if and only if $j_2\ge 4$. By \cref{prop:F.l=3,def:f}, 
\[
F(K)+f(R)=(2r_{j_2}-4)+(12-4r_{p_1})=8+2r_{j_2}-4r_{p_1}
>8+2-4\cdotp 2>0.
\]
When $j_2=3$, we have $K'=\varphi^{-1}(R)=3P22Q$ and $l_{K'}=2$. By \cref{prop:f<0.Y0}, we find $f(K')<0$. By \cref{prop:F.l=3,F.232}, 
\[
F(K)+F(K')=2(r_3-2)+\brk3{9-\frac{5}{2}r_{p_1}}=8-\frac{5}{2}r_{p_1}
\ge 8-\frac{5}{2}\cdotp 2>0.
\]
When $j_2=4$, we have $R=22Q4P$ with $(\abs{P}, \abs{Q})=(b+2, \,b-3)$. It is possible that $R$ coincides with the composition $M$ defined in \cref{lem:24>=4.i2=23}, i.e., $R=M$. By definition, this happens if and only if $p_1\ge 4$. Suppose that $p_1\ge 4$. Then by \cref{prop:F.l=3,lem:24>=4.i2=23}, 
\[
F(K)+\brk1{F(K'')+f(R)}=2(r_4-2)+\brk3{\frac{44}{3}-\frac{20}{3}r_{p_1}}
=\frac{20}{3}(2-r_{p_1})\ge 0.
\]
This completes the proof.
\end{proof}

Now we are in a position to give a proof for the $e$-positivity of $Y_0$.
\begin{lemma}\label{lem:e+:trinacria1.Y0}
The symmetric function $Y_0$ is $e$-positive.
\end{lemma}
\begin{proof}
For the decomposition in \cref{prop:f<0.Y0}, we further decompose $\mathcal B=\mathcal B_1\sqcup \mathcal B_2$, where
\[
\mathcal B_1=\{K\in\mathcal B\colon j_2\ge 3\}
\quad\text{and}\quad
\mathcal B_2=\{K\in\mathcal B\colon j_2=2\};
\]
decompose $\mathcal C=\mathcal C_1\sqcup\mathcal C_2\sqcup\mathcal C_3$, where
\[
\mathcal C_1
=\{K\in\mathcal C\colon j_1\not\in\{2, 4\}\}, 
\quad
\mathcal C_2
=\{K\in\mathcal C\colon j_1=2\}, 
\quad\text{and}\quad
\mathcal C_3
=\{K\in\mathcal C\colon j_1=4\};
\]
decompose $\mathcal D=\mathcal D_1\sqcup\dots\sqcup\mathcal D_4$, where
\[
\mathcal D_1
=\{K\in\mathcal D\colon j_2=2\}, 
\quad
\mathcal D_2
=\{K\in\mathcal D\colon j_2=3\}, 
\quad
\mathcal D_3
=\{K\in\mathcal D\colon j_2=4,\ i_2\ge 4\}, 
\]
and 
$\mathcal D_4
=\mathcal D\backslash \mathcal D_1\backslash \mathcal D_2\backslash \mathcal D_3$; and decompose 
$\mathcal E=\mathcal E_1\sqcup\dots\sqcup\mathcal E_5$, where
\begin{align*}
\mathcal E_1
&=\{K\in\mathcal E\colon (l_K, i_1, j_1)=(2, 4, 3),\ i_2\ge 3\}, 
\\
\mathcal E_2
&=\{K\in\mathcal E\colon (l_K, i_1, i_2)=(2, 4, 2),\ j_1\ge 4,\ k_3=2\}, 
\\
\mathcal E_3
&=\{K\in\mathcal E\colon (l_K, i_1, i_2)=(2, 4, 2),\ j_1\ge 4,\ k_3\ge 3\},\\
\mathcal E_4
&=\{K\in\mathcal E\colon (l_K, i_1, i_2)=(2, 4, 3),\ j_1\in\{4, 5\}\}, 
\end{align*}
and $\mathcal E_5
=\mathcal E\backslash \mathcal E_1\backslash \dotsm\backslash \mathcal E_4$.
Let 
\[
\mathcal F_1=\mathcal A\cup\mathcal B_1\cup\mathcal C_1\cup\mathcal D_1\cup\mathcal E_5
\quad\text{and}\quad
\mathcal F_2=\mathcal D_4\cup\mathcal E_3\cup\mathcal E_4.
\]
We define the aforementioned involution $\xi$ on $\mathcal K^-$ as follows:
\begin{itemize}
\item
define $\xi(K)=K$ to be the identity map for $K\in\mathcal F_1\cup\mathcal F_2$, 
\item
define $\xi\colon\mathcal C_2\to\mathcal C_2$ by $\xi(3P2Q)=3Q2P$, 
where $(\abs{P}, \abs{Q})=(b+1, \,b+1)$, 
\item
define $\xi\colon\mathcal C_3\to\mathcal E_1$ by $\xi(3P4Q)=4P3Q$, 
where $(\abs{P}, \abs{Q})=(b+1, \,b-1)$ and $p_1\ge 3$, 
\item
define $\xi\colon\mathcal D_2\to\mathcal B_2$ by $\xi(2P23Q)=3P22Q$, 
where $(\abs{P}, \abs{Q})=(b+2, \,b-2)$, 
\item
define $\xi\colon\mathcal D_3\to\mathcal E_2$ by $\xi(2P24Q)=422QP$, 
where $(\abs{P}, \abs{Q})=(b+2, \,b-3)$ and $p_1\ge 4$.
\end{itemize}

Now we follow the steps in our strategy.
For Step 1, let $K\in\mathcal F_1$.
If $l_K=2$, then $F(K)\ge 0$ by \cref{prop:F.l=2}; 
if $l_K=3$, then $F(K)\ge 0$ by \cref{prop:F.l=3}.
For $K\in\mathcal F_2$, let
\[
U=\begin{dcases*}
22Qj_2P, & if $K=2P2j_2Q\in\mathcal D_4$, \\
S4T, & if $K=4ST\in\mathcal E_3\cup \mathcal E_4$.
\end{dcases*}
\]
We shall show that $F(K)+f(U)>0$ and $U\in\mathcal H$.
\begin{itemize}
\item
If $K\in\mathcal D_4$, then 
$F(K)+f(U)>0$ by \cref{lem:322}.
Since $l_U=2$, we find $U\in\mathcal K^+$ by \cref{prop:f<0.Y0}.
Since $\abs{22Q}\le b+1$ and $\abs{22Qj_2}=b+5$, 
we find $U\in\mathcal H$.
\item
If $K\in\mathcal E_3\cup \mathcal E_4$, then
$F(K)+f(U)>0$ by \cref{lem:24>=4.i2=23}.
Since $l_U=2$ and $t_1\ge 4$, we find $U\in\mathcal K^+$ by \cref{prop:f<0.Y0}.
Since $\abs{S}=b+1$ and $\abs{S4}=b+5$, 
we find $U\in\mathcal H$.
\end{itemize}
This verifies Step 1 of our strategy.

For $K\in\mathcal K^-\backslash \mathcal F_1\backslash \mathcal F_2$, let $K'=\xi(K)$.
If $K\in\mathcal C_2\cup\mathcal C_3\cup\mathcal D_2$, 
then $F(K)+F(K')>0$ by \cref{lem:332,lem:334,lem:322}.
For $K\in\mathcal D_3$, we can write $K=2P24Q$, 
where $(\abs{P}, \abs{Q})=(b+2, \,b-3)$ and $p_1\ge 4$.
Let $V=22Q4P$. Then $F(K)+F(K')+f(V)\ge 0$ by \cref{lem:322}.
Since $l_V=2$, we find $f(V)\ge 0$ by \cref{prop:f<0.Y0}.
Since $\abs{22Q}=b+1$ and $\abs{22Q4}=b+5$, we find $V\in\mathcal H$.
This verifies Step 2.

For Step 3, we collect the compositions $U$ and $V$ used above:
\begin{enumerate}
\item
For $K\in\mathcal D_4$, we used the compositions $U=22Qj_2P$ with $(\abs{P},\, \abs{j_2Q})=(b+2, \,b+1)$; 
by the definition of $\mathcal D_4$, we have either $j_2\ge 5$, or ``$j_2=4$ and $p_1\le 3$''.
\item
For $K\in\mathcal E_3\cup\mathcal E_4$, 
we used the compositions $U=S4T$ with $(\abs{S}, \abs{T})=(b+1, \,b+2)$; 
by the definitions of $\mathcal E_3$ and $\mathcal E_4$, 
we have either ``$s_1=2$ and $s_2\ge 3$'' or ``$s_1=3$ and $t_1\in\{4, 5\}$''.
\item
For $K\in\mathcal D_3$, 
we used the compositions $V=22Q4P$ with $(\abs{P}, \abs{Q})=(b+2, \,b-3)$ and $p_1\ge 4$.
\end{enumerate}
It is direct to verify that the compositions $U$ and $V$ above are pairwise distinct.
This checks Step 3 of our strategy, and completes the proof.
\end{proof}

Now we can conclude a proof of \cref{thm:e+.trinacria1}.

\begin{proof}[Proof of \cref{thm:e+.trinacria1}]
Let $G=T_{(b+2)b2}$. 
Recall from \cref{prop:X.3sun:b+2.b.2} that $X_G=Y_2e^2+Y_1e_1+Y_0$.
By \cref{lem:e+:trinacria1.Y2,lem:e+:trinacria1.Y1,lem:e+:trinacria1.Y0},
we obtain the $e$-positivity of $Y_2$, $Y_1$ and $Y_0$, respectively.
Hence $X_G$ is $e$-positive.
\end{proof}

\bibliographystyle{abbrvnat}
\bibliography{csf}

\end{document}